\documentclass[12pt, a4paper, final]{article}
\usepackage{amsmath,amsthm,amsfonts,amssymb}
\usepackage{graphicx}
\usepackage{color}
\usepackage{url}
\usepackage{hyperref}
\usepackage{float}
\usepackage{import}
\usepackage{tikz}

\definecolor{lightgrey}{rgb}{0.5,0.5,0.5}
\definecolor{darkblue}{rgb}{0.0,0.0,0.7}
\definecolor{darkred}{rgb}{0.7,0.0,0.0}

\newcommand{\const}{\mathrm{const}}

\newcommand{\range}{\mathrm{range}}

\DeclareMathOperator{\opcurl}{\textnormal{\textbf{curl}}}
\DeclareMathOperator{\opdiv}{\textnormal{div}}
\DeclareMathOperator{\oprot}{\textnormal{curl}}

\newcommand{\complexi}{\textnormal{i}}

\newcommand{\fv}{{\bf f}}

\newcommand{\nv}{{\bf n}}

\newcommand{\tv}{{\bf t}}
\newcommand{\uv}{{\bf u}}
\newcommand{\vv}{{\bf v}}
\newcommand{\wv}{{\bf w}}
\newcommand{\xv}{{\bf x}}
\newcommand{\yv}{{\bf y}}
\newcommand{\zv}{{\bf z}}

\newcommand{\Av}{{\bf A}}
\newcommand{\Bv}{{\bf B}}

\newcommand{\Ev}{{\bf E}}

\newcommand{\Hv}{{\bf H}}

\newcommand{\Kv}{{\bf K}}
\newcommand{\Lv}{{\bf L}}

\newcommand{\Jv}{{\bf J}}

\newcommand{\Xv}{{\bf X}}

\newcommand{\lambdav}{\boldsymbol{\lambda}}

\theoremstyle{plain} 
\newtheorem{lemma}{Lemma}[section]
\newtheorem{theorem}[lemma]{Theorem}

\newtheorem{corollary}[lemma]{Corollary}

\theoremstyle{definition}
\newtheorem{definition}[lemma]{Definition}
\newtheorem{remark}[lemma]{Remark}

\title{A Tearing and Interconnecting Formulation for Magneto-Quasi-Statics}

\author{Clemens Pechstein\thanks{Dassault Syst\`emes Austria GmbH, Promenade 23, 4040 Linz, Austria, \texttt{clemens.pechstein@3ds.com}}}

\date{\today}

\begin{document}

\maketitle

\section{Introduction}

Tearing and interconnecting methods are non-overlapping domain decomposition methods for solving partial differential equations.
The first method of that kind was FETI (finite element tearing and interconnecting), introduced by Farhat and Roux
\cite{FarhatRoux:1991a,FarhatRoux:1994a}.
The original problem is reformulated using functions that are discontinuous across subdomain interfaces where the continuity is reinstalled by Lagrange multipliers. For a finite element setting with two subdomains, the basic formulation in matrix notation reads
\begin{align}
\label{eq:FETISPP}
  \begin{bmatrix} \Kv_1 & 0 & \Bv_1^T \\ 0 & \Kv_2 & \Bv_2 \\ \Bv_1 & \Bv_2 & 0 \end{bmatrix}
	\begin{bmatrix} \uv_1 \\ \uv_2 \\ \lambdav \end{bmatrix}
	\begin{bmatrix} \fv_1 \\ \fv_2 \\ 0 \end{bmatrix},
\end{align}
where $\Kv_k$ is the contribution of the stiffness matrix assembled on the elements of subdomain $\Omega_k$,
$\fv_k$ the contribution to the load vector, and $\Bv_k$ a zero-one matrix such that
$\sum_{k=1}^2 \Bv_k \uv_k = 0$ imposes the continuity constraint on the solution vectors $\uv_k$.
The terms $\Bv_k^T \lambdav$ can be interpreted as forces that ensure equilibrium in each subdomain.
Forming the Schur complement by sparse factorization, the saddle point system~\eqref{eq:FETISPP} can be reduced to a problem in the Lagrange multipliers only,
which is then solved iteratively, typically involving preconditioners.
In case the local matrices $\Kv_k$ are not invertible, one either has to work with their kernel explicitly (as in the original FETI method)
or use a technique called dual-primal FETI (FETI-DP) \cite{FarhatLesoinneLeTallecPiersonRixen:2001a}, see also \cite{ToselliWidlund:Book}.
In the latter approach, for a carefully selected small set of \emph{primal} degrees of freedom (DOFs), no tearing and interconnecting takes place. The local operators then need to be invertible only on the subspace where the primal DOFs vanish and one has to solve in addition a global problem of the size of the primal DOFs.

Whereas FETI and FETI-DP were initially proposed for finite-element discretized problems from structural mechanics and mainly used to speed up and parallelize the solution process,
the basic technology has with the years been carried over to boundary element discretizations \cite{LangerSteinbach:2003a,Of:PhD,Pechstein:FETIBook}
(`BETI')
and isogeometric analysis \cite{KleissPechsteinJuettlerTomar:2012,SchneckenleitnerTakacs:2020} (`IETI').
Also, extensions to problems in acoustics \cite{FarhatMacedoLesoinneRouxMagoulesDeLaBourdonnaie:2000a},
positive definite problems in H(curl) \cite{Toselli:2006,DohrmannWidlund:2016a}
and time-harmonic Maxwell's equations \cite{Vouvakis:PhD,Windisch:PhD,Paraschos:PhD} have been developed,
see also \cite{Pechstein:2023} and references therein.
Closely related methods are balancing Neumann-Neumann \cite{MandelBrezina:1996a}
and balancing domain decomposition by constraints (BDDC) \cite{Dohrmann:2003a,MandelDohrmannTezaur:2005a},
see also \cite{ToselliWidlund:Book,PechsteinDohrmann:2017a} and references therein.

\medskip

The \emph{magneto-quasi-static} or eddy current model is a low-frequent approximation of the Maxwell equations,
neglecting any displacement currents and therefore simplifying an originally hyperbolic problem to a parabolic one \cite{AlonsoValli:Book}.
If all fields are time-invariant,
one is left with the \emph{magneto-static} problem
\begin{align}
  \opcurl\Hv = \Jv_i, \qquad \opdiv(\Bv) = 0, \qquad \Bv = \mu\Hv,
\end{align}
where $\Bv$, $\Hv$ are the magnetic fields, $\Jv_i$ the impressed current density, and $\mu$ the magnetic permeability.
Usually, this problem is reformulated using a magnetic vector potential $\Av$ such that $\Bv = \opcurl\Av$, reading
\begin{align}
\label{eq:MStaticA}
  \opcurl(\mu^{-1} \opcurl\Av) = \Jv_i\,.
\end{align}
Here $\Av$ is only unique up to functions in the kernel of the curl operator, in particular gradient fields.
Classically, to make $\Av$ unique in formulations or numerical computations, one applies a \emph{gauge condition} as a constraint,
e.g.\ the Colomb gauge $\opdiv\Av = 0$.
For finite elements (and also isogeometric analysis, cf.\ \cite{KapidaniMerkelSchoepsVazquez:2022}),
a very efficient and practical technique is the (generalized) \emph{tree-cotree} gauging
\cite{AlbaneseRubinacci:1988,MangesCendes:1995}.
There, the edges of the mesh are split into a spanning tree and the remaining edges, called \emph{cotree edges}.
For a lowest-order N\'ed\'elec discretization,
the curl-curl matrix is invertible on the subspace spanned by the basis functions associated with the cotree edges, see also \cite{RapettiAlonsoLosSantos:2022}.

\medskip

A tearing and interconnecting formulation for magneto-\emph{statics} has been proposed by Mally et al.\
\cite{MallyKapidaniMerkelSchoepsVaquez:2025}.
When applying the tearing step to \eqref{eq:MStaticA}, the local subdomain operators can happen to lose their invertibility property.
This is because the original tree restricted to a subdomain is not necessarily a spanning tree anymore.
However, this problem can be overcome by using small number of primal unknowns in a dual-primal framework,
for details see \cite{MallyKapidaniMerkelSchoepsVaquez:2025}.

\medskip

The magneto-\emph{quasi-static} problem may be formulated as
\begin{align}
\label{eq:MQStaticA}
  \opcurl(\mu^{-1} \opcurl\Av) + \complexi\omega \sigma\Av = \Jv_i\,,
\end{align}
where $\omega$ is the angular frequency in the time-harmonic setting and $\sigma$ the electric conductivity.
The $\Av$-field is non-unique only in the non-conducting region, i.e., the part of the domain where $\sigma$ vanishes.

If $\sigma > 0$ on the \emph{entire} computational domain, say $\Omega$, we do not need any gauge condition,
and a tearing and interconnecting machinery can be applied out of the box.
If, however, the domain $\Omega$ splits into a conducting domain $\Omega_C$ and an insulating (non-conducting) domain $\Omega_I$,
then the following problem arises.
Consider a tree-cotree splitting for the mesh on $\Omega$. We need all \emph{tree} edges associated with $\overline{\Omega_I}$
(particularly on the interface $\Gamma$ between $\Omega_C$ and $\Omega_I$) to be able to represent the full lowest-order N\'ed\'elec space on $\Omega_C$.
However, the basis functions associated with tree edges on the \emph{interface} have support in the non-conducting subdomain as well.
Consequently, if one introduces Lagrange multipliers to couple tree edge DOFs across the interface $\Gamma$,
one ends up with a large kernel of the operator in the insulating domain. Using the idea from \cite{MallyKapidaniMerkelSchoepsVaquez:2025}
in such a context results in a large number of primal DOFs and prevents any iterative method based on that formulation from being efficient.

\medskip

The purpose of the note at hand is to show that, in a certain sence,
the coupling of tree edge DOFs across the interface between conducting and insulating subdomain is not necessary.
Instead, one can work with an $\Av$-field that is only \emph{partially} continuous.
Rather than using a tree-cotree splitting,
the key idea is based on a splitting into gradient fields and a complementary subspace on which the curl-operator is invertible,
such as in \cite{Jochum:PhD}, see also \cite{HiptmairKraemerOstrowski:2008,EllerReitzingerSchoepsZaglmayr:2017}.
In a finite element setting for the lowest-order N\'ed\'elec space, this complementary space can be chosen as the space spanned by the cotree basis functions,
but the following construction works for any complementary space. Also, it is presented for both the continuous as well as the general discrete setting.
With the gradient-based splitting, it is only necessary to impose continuity constraints on traces / DOFs associated with the complementary subspace,
but not with the gradient part.
In the insulating subdomain, no DOFs associated with the gradient part are present.
One can think of the gradient part on the insulating region being simply set to zero, which is why the $\Av$-field solving the formulation is discontinuous.
Although this sounds dangerous at first glance, it turns out that under a mild assumption on the splitting related to the interface,
this nevertheless provides a correct solution to the original equation, with a continuous magnetic field $\Bv = \opcurl\Av$.
In addition, both subdomain operators are invertible.

\medskip

The author is grateful to Sebastian Sch\"ops, Mario Mally, and Rafael V\'asquez for bringing this problem to his attention
and in particular would like to thank Rafael V\'asquez for the fruitful discussion at RICAM Linz in October 2025.

\section{Geometric setting and single-domain formulation}

\subsection{Geometric setting and assumptions}
\label{sect:geoAss}

Let $\Omega \subset \mathbb{R}^3$ be a bounded Lipschitz domain, split into two non-overlapping subdomains $\Omega_C$ (conductor) and $\Omega_I$ (insulator),
and let $\Gamma = \partial\Omega_C \cap \partial\Omega_I$ denote the interface.
We assume that both $\Omega_C$ and $\Omega_I$ are Lipschitz domains and have simple topology,
in particular simply connected with simply connected boundary.

Let $\sigma \in L^\infty(\Omega)$ denote the conductivity coefficient and assume that it is uniformly positive in $\Omega_C$ and zero in $\Omega_I$.
The magnetic permeability $\mu$ is assumed to be $L^\infty(\Omega)$ and uniformly positive in $\Omega$.

For simplicity, we use magnetic boundary conditions on all of $\partial\Omega$,
but the following can be extended to the case of `Dirichlet' (PEC) boundary conditions on $\partial\Omega$ or on a part of $\partial\Omega$.

\subsection{A gradient-explicit formulation of magneto-quasi-statics}

The magneto-quasi-static equations read
\begin{align}
  \left. \begin{array}{rclrcl}
    \opcurl\Hv & = & \Jv_i + \sigma\Ev, \quad &  \opdiv\Bv & = & 0 \\
	  \opcurl\Ev & = & -\complexi\omega \Bv,
	  \qquad & \Bv & = & \mu\Hv
	\end{array} \right\} \qquad & \text{in } \Omega,\\
	\begin{array}{rcl}
	  \Hv \times \nv & = & \Jv_s
	\end{array} \qquad	& \text{on } \partial\Omega,
\end{align}
where $\Jv_i$ is the impressed current density and $\Jv_s$ a surface current, cf.\ \cite{AlonsoValli:Book}.
Note that there is no term of $\varepsilon \Ev$ in the very first equation.
Under the assumption that $\Bv \in \Lv^2(\Omega)$, we can deduce from $\opdiv\Bv = 0$
the existence of a vector potential $\Av \in \Hv(\opcurl,\Omega)$ such that
\begin{align}
\label{eq:BcurlA}
  \Bv = \opcurl\Av,
\end{align}
see e.g., \cite[Thm.~3.38]{Monk:Book}.
Assuming that $\Ev \in \Lv^2(\Omega)$,
we find from $\opcurl(\Ev+\complexi\omega\Av) = 0$ that there exists a scalar potential $\varphi \in H^1(\Omega)$ such that
\begin{align}
\label{eq:EphiA}
  \Ev = -\nabla\varphi - \complexi\omega\Av \qquad \text{in } \Omega,
\end{align}
see e.g., \cite[Thm.~3.37]{Monk:Book}.
After inserting the constitutive law $\Hv = \mu^{-1} \Bv$ and the two relations above, the equation $\opcurl\Hv = \Jv_i + \sigma\Ev$ changes to
\begin{align}
  \opcurl(\mu^{-1} \opcurl\Av) = \Jv_i + \sigma(-\nabla\varphi - \complexi\omega\Av) \qquad \text{in } \Omega,
\end{align}
and the boundary condition to
\begin{align}
  \mu^{-1} \opcurl\Av \times \nv = \Jv_s \qquad \text{on } \partial\Omega.
\end{align}
The variational formulation reads as follows.
Find $\Av \in \Hv(\opcurl, \Omega)$ and $\varphi \in H^1(\Omega)$ such that
\begin{align}
\label{eq:MQSunconstrained}
  \int_\Omega \mu^{-1} \opcurl\Av \cdot \opcurl \vv \, dx + \int_{\Omega_C} \sigma (\nabla\varphi + \complexi\omega \Av) \cdot \vv \, dx
	= \langle \Jv, \vv \rangle_\Omega
\end{align}
for all $\vv \in \Hv(\opcurl,\Omega)$.
Above, we have combined $\Jv_i$ and $\Jv_s$ into a distributional current $\Jv \in \Hv(\opcurl,\Omega)^*$
and $\langle \cdot, \cdot \rangle_\Omega$ denotes the dual pairing on $\Hv(\opcurl,\Omega)^* \times \Hv(\opcurl,\Omega)$,

Equation~\eqref{eq:MQSunconstrained} is only solvable under the compatibility condition that $\Jv$ is solenoidal outside of the conducting domain, i.e.,
\begin{align}
\label{eq:JsoleniodalI}
  \langle \Jv, \nabla q_I \rangle_\Omega = 0 \qquad \forall q_I \in H^1_\Gamma(\Omega_I),
\end{align}
where $H^1_\Gamma(\Omega_I)$ is the subspace of $H^1(\Omega_I)$ functions vanishing on $\Gamma$
(and whose extension by zero is therefore in $H^1(\Omega)$).

The electric field is defined via \eqref{eq:EphiA}. However, this field is only well-defined in the conducting domain $\Omega_C$.
Outside, $\varphi$ can be chosen quite arbitrarily without altering \eqref{eq:MQSunconstrained}.
In order to make $\varphi$ unique outside the conducting domain, we can, e.g., use
\begin{align}
  \int_{\Omega_I} \nabla \varphi \cdot \nabla q_I \, dx = 0 \qquad \forall q_I \in H^1_\Gamma(\Omega_I)
\end{align}
or any other appropriate condition.
But no matter how this is done, the values of $\varphi$ outside of $\Omega_C$ do not play any role in \eqref{eq:MQSunconstrained},
which is why we can restrict $\varphi$ to $H^1(\Omega_C)$ from the start.
To obtain an (arbitrary) $\Ev$-field in the whole domain, we can then take this $\varphi$ and extend it to a function in $H^1(\Omega)$. 

The splitting \eqref{eq:EphiA} in $\overline{\Omega_C}$ is not unique either.
And even if this splitting is made unique, looking at \eqref{eq:MQSunconstrained},
still the $\Av$-field outside of $\Omega_C$ is only unique up to $\nabla H^1_\Gamma(\Omega_I)$.
In order to fix $\Av$, we can, e.g., use the Coulomb gauge
\begin{align}
  \int_\Omega \Av \cdot \nabla q \, dx = 0 \qquad \forall q \in H^1(\Omega),
\end{align}
or we could also use two separate conditions,
\begin{alignat}{2}
  \int_{\Omega_C} \Av \cdot \nabla q_C \, dx & = 0 \qquad && \forall q_C \in H^1(\Omega_C),\\
	\int_{\Omega_I} \Av \cdot \nabla q_I \, dx & = 0 \qquad && \forall q_I \in H^1_\Gamma(\Omega_I),
\end{alignat}
or we can say that $\Av$ is from any complementary space $\mathcal{V}$ such that
\begin{align}
  \mathcal{H} := \Hv(\opcurl,\Omega) = \nabla \underbrace{ H^1(\Omega)_{/\const} }_{=: \mathcal{U}} \oplus \mathcal{V},
\end{align}
where this splitting is \emph{direct} but not necessarily orthogonal.
The key properties of the subspaces $\mathcal{U}$ and $\mathcal{V}$ are that
\begin{enumerate}
\item[(i)] $\opcurl \uv = 0$ for all $\uv \in \nabla \mathcal{U}$,
\item[(ii)] the bilinear form $\int_{\Omega} \mu^{-1} \opcurl \vv \cdot \opcurl \wv \, dx$ is coercive on $\mathcal{V}$
  (positively bounded from below) and therefore associated with an invertible operator.
\end{enumerate}
In \eqref{eq:MQSunconstrained}, the test function $\vv$ is still from the entire space $\Hv(\opcurl,\Omega)$.
But we may of course test with $\vv = \wv \in \mathcal{V}$ and $\vv = \nabla q$ with $q \in \mathcal{U}$ separately:
Find $\Av \in \mathcal{V}$ and $\varphi \in H^1(\Omega_C)_{/\const}$ such that
\begin{align}
  \int_\Omega \mu^{-1} \opcurl\Av \cdot \opcurl \wv \, dx + \int_{\Omega_C} \sigma (\nabla\varphi + \complexi\omega \Av) \cdot \wv \, dx
	& = \langle \Jv, \wv \rangle_\Omega \\
\label{eq:preMQS:2}
	\int_{\Omega_C} \sigma (\nabla\varphi + \complexi\omega \Av) \cdot \nabla q \, dx
	& = \langle \Jv, \nabla q \rangle_\Omega
\end{align}
for all $\wv \in \mathcal{V}$ and $q \in H^1(\Omega)_{/\const}$.
Apparently, the last equation is not altered if we test with $q \in H^1(\Omega)$, including globally constant functions.
Since both left- and right-hand side of \eqref{eq:preMQS:2} vanish for $q \in H^1_\Gamma(\Omega)$
(see \eqref{eq:JsoleniodalI}), we can restrict the test functions $q$
to any complementary space $\mathcal{Q}$ such that $H^1(\Omega) = \mathcal{Q} \oplus H^1_\Gamma(\Omega_I)$.
All such spaces $\mathcal{Q}$ are isomorphic to $H^1(\Omega_C)$ but one needs to fix a particular extension from $\Omega_C$ to $\Omega$.
Let $E$ denote a bounded linear extension operator, then we can simply set $\mathcal{Q} = E(H^1(\Omega_C))$.
Finally, we can factor out again the global constants  in $\varphi$ on $\overline{\Omega_C}$.

To summarize: find $\Av \in \mathcal{V}$ and $\varphi \in \mathcal{U}_C := H^1(\Omega_C)_{/\const}$ such that
\begin{align}
\label{eq:MQS:1}
  \int_\Omega \mu^{-1} \opcurl\Av \cdot \opcurl \wv \, dx + \int_{\Omega_C} \sigma (\nabla\varphi + \complexi\omega \Av) \cdot \wv \, dx
	& = \langle \Jv, \wv \rangle_\Omega \\
\label{eq:MQS:2}
	\int_{\Omega_C} \sigma (\nabla\varphi + \complexi\omega \Av) \cdot \nabla q \, dx
	& = \langle \Jv, \nabla(E q) \rangle_\Omega
\end{align}
for all $\wv \in \mathcal{V}$ and $q \in \mathcal{U}_C$.
A similar formulation can be found in \cite[Sect.~6.1]{AlonsoValli:Book}.
Because of \eqref{eq:JsoleniodalI} the right-hand side is independent of the particular extension operator $E$.
Using an operator notation, the equations~\eqref{eq:MQS:1}--\eqref{eq:MQS:2} can be rewritten as follows (with a more explicit structure):
find $(\Av, \varphi) \in \mathcal{V} \times \mathcal{U}_C$ such that
\begin{align}
\label{eq:MQS:operator}
  \begin{bmatrix} K + \complexi\omega M & S^T \\ \complexi\omega S & C \end{bmatrix}
	\begin{bmatrix} \Av \\ \varphi \end{bmatrix}
	= \begin{bmatrix} J \\ j \end{bmatrix}
\end{align}
with
\begin{align}
\begin{alignedat}{2}
  \langle K \Av, \wv \rangle & := \int_\Omega \mu^{-1} \opcurl\Av \cdot \opcurl \wv \, dx \qquad && \forall \Av, \wv \in \mathcal{V},\\
	\langle M \Av, \wv \rangle & := \int_{\Omega_C} \sigma \Av \cdot \wv \, dx              \qquad && \forall \Av, \wv \in \mathcal{V},\\
	\langle S \Av, q \rangle & := \int_{\Omega_C} \sigma \Av \cdot \nabla q \, dx           \qquad && \forall \Av \in \mathcal{V}, q \in \mathcal{U}_C,\\
	\langle C \varphi, q \rangle & := \int_{\Omega_C} \sigma \nabla\varphi \cdot \nabla q \, dx \qquad && \forall \varphi, q \in \mathcal{U}_C,\\
	\langle J, \wv \rangle & := \langle \Jv, \wv \rangle_\Omega                                 \qquad && \forall \wv \in \mathcal{V},\\
\label{eq:jDef}
	\langle j, q \rangle & := \langle \Jv, \nabla(E q) \rangle_\Omega                           \qquad && \forall q \in \mathcal{U}_C.
\end{alignedat}
\end{align}
The advantage of this formulation (for our purposes, compared to other magneto-quasistatic formulations)
is that the operator $K$ is invertible and that the gradient fields $\nabla \varphi$,
spanning the kernel of the curl operator, are separated.
Note that by the substitution $\varphi = \complexi\omega \widetilde\varphi$ (assuming $\omega > 0$),
one can make \eqref{eq:MQS:operator} symmetric.

\subsection{Relation to a more conventional formulation}

Given that $\omega > 0$, we may start again with \eqref{eq:MQSunconstrained} and set $\varphi = 0$ (like a gauge condition), such that
\begin{align}
  \Ev = - \complexi\omega\Av,
\end{align}
compared to \eqref{eq:EphiA}.
The variational formulation then reads: find $\Av \in \Hv(\opcurl,\Omega)$ such that
\begin{align}
  \int_\Omega \mu^{-1} \opcurl\Av \cdot \opcurl\vv \, dx + \complexi\omega \int_{\Omega_C} \sigma \Av \cdot \vv \, dx = \langle \Jv, \vv \rangle_\Omega
\end{align}
for all $\vv \in \Hv(\opcurl,\Omega)$.
Here, again $\Jv \in \Hv(\opcurl,\Omega)^*$ with \eqref{eq:JsoleniodalI}.
The solution $\Av$, however, is not unique; one can add any function from $\nabla H^1_\Gamma(\Omega_I)$
(implicitly extended by zero from $\Omega_I$ to $\Omega$).
To make $\Av$ unique, we use some complementary space $\mathcal{W}$ such that
\begin{align}
  \Hv(\opcurl,\Omega) = \nabla H^1_\Gamma(\Omega_I) \oplus \mathcal{W}.
\end{align}
If we split $\mathcal{W} = \mathcal{Q} \oplus \mathcal{V}$ with $\mathcal{Q} = \nabla E(H^1(\Omega_C)_{/\const})$
and with a suitable subspace $\mathcal{V}$,
then the formulation rewrites: find $\Av_\mathcal{Q} \in \mathcal{Q}$ and $\Av_\mathcal{V} \in \mathcal{V}$ such that
\begin{align}
  \int_\Omega \mu^{-1} \opcurl\Av_\mathcal{V} \cdot \opcurl\wv \, dx + \complexi\omega \int_{\Omega_C} \sigma (\Av_\mathcal{Q} + \Av_\mathcal{V}) \cdot \wv \, dx
	& = \langle \Jv, \wv \rangle_\Omega \\
	\complexi\omega \int_{\Omega_C} \sigma (\Av_\mathcal{Q} + \Av_\mathcal{V}) \cdot \yv \, dx & = \langle \Jv, \yv \rangle_\Omega
\end{align}
for all $\yv \in \mathcal{Q}$.
Setting $\Av_\mathcal{Q} = \frac{1}{\complexi\omega} \nabla(E \varphi)$ and $\yv = \nabla(E q)$ with $\varphi$, $q \in H^1(\Omega_C)_{/\const}$
yields again \eqref{eq:MQS:1}--\eqref{eq:MQS:2}.

\subsection{Discrete setting}

In order to discretize formulation~\eqref{eq:MQS:operator}, we simply use a subspace $\mathcal{H}_h \subset \mathcal{H} = \Hv(\opcurl,\Omega)$
and subspaces $\mathcal{U}_h \subset \mathcal{U}$ and $\mathcal{V}_h \subset \mathcal{H}_h$,
and set $\mathcal{U}_{C,h} = \mathcal{U}_{h|\Omega_C}$, such that
\begin{enumerate}
\item[(i)] $\mathcal{H}_h = \nabla \mathcal{U}_h \oplus \mathcal{V}_h$,
\item[(ii)] $\opcurl \uv_h = 0$ for all $\uv \in \nabla \mathcal{U}_h$,
\item[(iii)] the bilinear form $\int_{\Omega} \mu^{-1} \opcurl \vv_h \cdot \opcurl \wv_h \, dx$ is coercive
  (positively bounded from below) on $\mathcal{V}_h$.
\end{enumerate}
Then the Galerkin projection of~\eqref{eq:MQS:operator} has exactly the same form, just that $\mathcal{U}_C$ and $\mathcal{V}$
are replaced by $\mathcal{U}_{C,h}$, $\mathcal{V}_h$.

As an example, if $\mathcal{H}_h$ is the lowest-order space N\'ed\'elec space and $\mathcal{U}_h$ the space of piecewise linear functions
(with the global constant factored out),
then we can choose $\mathcal{V}_h$ as the space spanned by the basis functions associated with cotree edges, cf.\ \cite{Jochum:PhD}.

Note that the original condition \eqref{eq:JsoleniodalI} has to be replaced by
\begin{align}
\label{eq:JsoleniodalII}
  \langle \Jv, \nabla q \rangle_{\Omega} = 0 \qquad \forall q \in \mathcal{U}_h \text{ with } q_{|\overline{\Omega_C}} = 0,
\end{align}
i.e., the given current must be discrete divergence-free outside the conducting domain.
In the sequel, the subscript $h$ is omitted and the continuous and discrete case are treated uniformly.

\section{Auxiliary results}

In the following, for any vector field $\vv$ defined on $\Omega_C$, $\Omega_I$, or $\Gamma$,
\[
  \vv_\tau := \nv \times \vv \times \nv, \qquad \vv_n := \vv \cdot \nv
\]
denote the tangential and normal component on $\Gamma$, while $\nv$ is a unit normal vector on $\Gamma$ with fixed orientation.

\begin{lemma}
  Let $u \in L^2(\Omega)$ with $u_C := u_{|\Omega_C} \in H^1(\Omega_C)$ and $u_I := u_{|\Omega_I} \in H^1(\Omega_I)$
	and define
	\[
	  (\widetilde{\nabla} u)(\xv) := \begin{cases}
		  (\nabla u_C)(\xv) & \xv \in \Omega_C\,,\\
			(\nabla u_I)(\xv) & \xv \in \Omega_I\,.
		\end{cases}
	\]
	Then $\widetilde{\nabla} u \in \nabla H^1(\Omega) \subset \Hv(\opcurl,\Omega)$ if any only if
	\[
	  \nabla_\Gamma u_C - \nabla_\Gamma u_I = 0 \quad \text{on } \Gamma,
	\]
	where $\nabla_\Gamma$ is the surface gradient.
	If $\Gamma$ is connected, the last condition is equivalent to
	\[
	  u_C - u_I = \const \quad \text{on } \Gamma.
	\]
\end{lemma}
\begin{proof}
  Apparently, the vector field $\widetilde{\nabla} u$ is piecewise in $\Hv(\opcurl)$.
	The same vector field is globally in $\Hv(\opcurl,\Omega)$ if and only if the tangential trace on the interface has to be continuous, i.e.,
	\[
	  (\nabla u_C)_\tau = (\nabla u_I)_\tau \qquad \text{on } \Gamma.
	\]
	Since $(\nabla v)_\tau = \nabla_\Gamma v$ this is equivalent to
	\[
	  \nabla_\Gamma (u_C - u_I) = 0 \quad \text{on } \Gamma.
	\]
	If $\Gamma$ is connected, this means that $u_C - u_I$ is constant.
\end{proof}

\begin{lemma}
\label{lem:pwHCurl}
  Let $\vv \in \Lv^2(\Omega)$ with $\vv_C := \vv_{|\Omega_C} \in \Hv(\opcurl,\Omega_C)$ and $\vv_I := \vv_{|\Omega_I} \in \Hv(\opcurl,\Omega_I)$
	and define
	\[
	  (\widetilde{\opcurl}\, \vv)(\xv) := \begin{cases}
		  (\opcurl \vv_C)(\xv) & \xv \in \Omega_C\,,\\
			(\opcurl \vv_I)(\xv) & \xv \in \Omega_I\,.
		\end{cases}
	\]
	Then $\widetilde{\opcurl}\, \vv \in \opcurl \Hv(\opcurl,\Omega) \subset \Hv(\opdiv, \Omega)$ if and only if
	\[
	  \oprot_\Gamma \vv_{C,\tau} - \oprot_\Gamma \vv_{I,\tau} = 0 \quad \text{on } \Gamma,
	\]
	where $\oprot_\Gamma$ is the scalar surface curl of a tangential vector field.
	If $\Gamma$ is simply connected, the last condition is equivalent to
	\[
	  \vv_{C,\tau} - \vv_{I,\tau} = \nabla_\Gamma \psi
		\quad \text{for some } \psi \in H^{1/2}(\Gamma).
	\]
\end{lemma}
\begin{proof}
  Apparently, the vector field $\widetilde{\opcurl} \vv$ is piecewise in $\Hv(\opdiv)$.
	The same vector field is globally in $\Hv(\opdiv,\Omega)$ if and only if the normal trace on the interface has to be continuous, i.e.,
	\[
	  (\opcurl \vv_C)_n = (\opcurl \vv_I)_n = 0 \qquad \text{on } \Gamma
	\]
	in the sense of $H^{1/2}(\Gamma)^*$.
	Since $(\opcurl \vv)_n = \pm \oprot_\Gamma \vv_\tau$ this is equivalent to
	\[
	  \oprot_\Gamma (\vv_{C,\tau} - \vv_{I,\tau}) = 0 \quad \text{on } \Gamma.
	\]
	If $\Gamma$ is simply connected, this means that $\vv_{C,\tau} - \vv_{I,\tau}$ is a surface gradient field.
\end{proof}

In the following, let $\Hv(\opcurl_\Gamma, \Gamma) := \{ \vv_{\tau|\Gamma} \colon \vv \in \Hv(\opcurl,\Omega) \}$
be the tangential trace space of $\Hv(\opcurl, \Omega)$
such that the trace operator $\Hv(\opcurl, \Omega) \ni \vv \mapsto \vv_\tau \in \Hv(\opcurl_\Gamma, \Gamma)$
is surjective and has a continuous right-inverse (see \cite{BuffaCostabelSheen:2002}).
Note that here $\Omega$ can be replaced by $\Omega_C$ or $\Omega_I$, resulting in the same trace space.

\begin{definition}
\label{def:compatible}
  Let $\mathcal{H}$ either be $\Hv(\opcurl, \Omega)$ or a discrete subspace and let
	$\mathcal{H} = \nabla\mathcal{U} \oplus \mathcal{V}$ be a direct space splitting.
	We call this splitting \emph{compatible} with the geometric decomposition of $\Omega$ into $(\Omega_C, \Omega_I, \Gamma)$
	if the following conditions hold.
	\begin{enumerate}
	\item[(i)]
		Let $\mathcal{H}_\Gamma := \{ \wv_{\tau|\Gamma} \colon \wv \in \mathcal{H} \}$,
		$\mathcal{U}_\Gamma := \{ u_{|\Gamma} \colon u \in \mathcal{U} \}$, and
		$\mathcal{V}_\Gamma := \{ \vv_{\tau|\Gamma} \colon \vv \in \mathcal{V} \}$ denote the trace spaces of $\mathcal{H}$, $\mathcal{U}$, $\mathcal{V}$
		respectively,
		then
		\[
		  \mathcal{H}_\Gamma = \nabla \mathcal{U}_\Gamma \oplus \mathcal{V}_\Gamma\,.
		\]
	\item[(ii)]
	  Let $\mathcal{H}_C := \{ \wv_{|\Omega_C} \colon \wv \in \mathcal{H} \}$,
		$\mathcal{U}_C := \{ u_{|\Omega_C} \colon u \in \mathcal{U} \}$, and
		$\mathcal{V}_C := \{ \vv_{|\Omega_C} \colon \vv \in \mathcal{V} \}$
		denote the restrictions of $\mathcal{H}$, $\mathcal{U}$, $\mathcal{V}$
		to $\Omega_C$, then
	  \[
		  \mathcal{H}_C = \nabla \mathcal{U}_C \oplus \mathcal{V}_C\,.
		\]
	\item[(iii)]
	  Let $\mathcal{H}_I$, $\mathcal{U}_I$, and $\mathcal{V}_I$ denote the restrictions of $\mathcal{H}$, $\mathcal{U}$, $\mathcal{V}$
		to $\Omega_I$, respectively, then
	  \[
		  \mathcal{H}_I = \nabla \mathcal{U}_I \oplus \mathcal{V}_I\,.
		\]
	\end{enumerate}
\end{definition}

\begin{remark}
  In the case where $\mathcal{H} = \Hv(\opcurl,\Omega)$, the space $\mathcal{U}_\Gamma$ is $H^{1/2}(\Gamma)_{/\const}$.
	If $\mathcal{H}$ is the lowest-order N\'ed\'elec space,
	we can choose $\mathcal{U}_\Gamma$ as the piecewise linear subspace of $H^{1/2}(\Gamma)_{/\const}$
	and use a tree-cotree splitting on $\Gamma$ that is continued to a tree-cotree splitting on $\Omega_C$ and $\Omega_I$, separately.
\end{remark}

\begin{lemma}
\label{lem:compatSplit}
  Let $\Omega_C$, $\Omega_I$, and $\Gamma$ meet the assumptions stated in Sect.~\ref{sect:geoAss}.
  Let $\Hv(\opcurl, \Omega) \supseteq \mathcal{H} = \nabla \mathcal{U} \oplus \mathcal{V}$
	be a direct splitting, compatible in the sense of Definition~\ref{def:compatible}
	and let $\vv \in \Lv^2(\Omega)$ be given by $\vv_{|\Omega_C} = \vv_C \in \mathcal{V}_C$ and $\vv_{|\Omega_I} = \vv_I \in \mathcal{V}_I$.
	Then the following three conditions are equivalent:
	\begin{enumerate}
	\item[(i)]
	  $\widetilde{\opcurl}\, \vv \in \opcurl \Hv(\opcurl,\Omega) \subset \Hv(\opdiv, \Omega)$,
	\item[(ii)]
	  $\vv_{C,\tau} - \vv_{I,\tau} = 0$ on $\Gamma$,
	\item[(iii)]
	  $\vv \in \mathcal{V}$.
	\end{enumerate}
\end{lemma}
\begin{proof}
  The equivalence between (i) and (ii) follows from Lemma~\ref{lem:pwHCurl}
	and the fact that there are no surface gradients in $\mathcal{V}_\Gamma$.
	The equivalence between (ii) and (iii) is a well-known fact and holds even for piecewise $\Hv(\opcurl)$ functions.
\end{proof}

\begin{corollary}
\label{cor:curl}
  Let $\Omega_C$, $\Omega_I$, and $\Gamma$ meet the assumptions stated in Sect.~\ref{sect:geoAss}
  and let $\Hv(\opcurl, \Omega) \supseteq \mathcal{H} = \nabla \mathcal{U} \oplus \mathcal{V}$
  be a direct splitting, compatible in the sense of Definition~\ref{def:compatible}.
  Let $\vv \in \Lv^2(\Omega)$ be given by $\vv_{|\Omega_C} = \vv_C = \nabla \varphi_C + \zv_C$ and 
	$\vv_{|\Omega_I} = \vv_I = \zv_I + \nabla \varphi_I$.
	Then the following two conditions are equivalent:
	\begin{enumerate}
	\item[(i)]
	  $\widetilde{\opcurl}\, \vv \in \opcurl \Hv(\opcurl,\Omega) \subset \Hv(\opdiv, \Omega)$,
	\item[(ii)]
	  $\zv_{C,\tau} - \zv_{I,\tau} = 0$ on $\Gamma$.
	\end{enumerate}
\end{corollary}
\begin{proof}
  We only need to apply Lemma~\ref{lem:compatSplit} to $\zv \in \Lv^2(\Omega)$ defined piecewise by $\zv_C$ and $\zv_I$
	and use that $\opcurl\nabla = 0$ on each subdomain.
\end{proof}

The above result shows that if we use the splitting of Definition~\ref{def:compatible} and couple only the traces of the $\mathcal{V}$-component,
but leave the gradient component discontinuous, we still have $\opcurl \vv$ globally in $\Hv(\opdiv)$.
So if $\vv = \Av$ is the magnetic vector potential, the corresponding $\Bv$-field will be continuous across the subdomain interface $\Gamma$.

\begin{corollary}
\label{cor:X}
  Let $\Hv(\opcurl, \Omega) \supseteq \mathcal{H} = \nabla \mathcal{U} \oplus \mathcal{V}$ be a direct splitting, compatible in the sense of Definition~\ref{def:compatible}
	and define
	\[
	  \mathcal{X} := \mathcal{V}_C \times \mathcal{V}_I
	\]
	which, equipped with the norm $\| (\vv_C, \vv_I) \} := \big( \| \vv_C \|_{\Hv(\opcurl, \Omega_C)}^2 + \| \vv_I \|_{\Hv(\opcurl,\Omega_I)}^2 \big)^{1/2}$,
	is a Hilbert space.
  Then $\mathcal{V}$ is isomorphic to the subspace
  \[
	  \widehat{\mathcal{X}} := \left\{ (\vv_C, \vv_I) \in \mathcal{X} \colon \vv_{C,\tau} - \vv_{I,\tau} = 0 \text{ on } \Gamma \right\}.
	\]
	More precisely, if
	\[
	  R \colon \mathcal{V} \to \mathcal{X} \colon \vv \mapsto (\vv_{|\Omega_C}, \vv_{|\Omega_I})
	\]
	denotes the `tearing' operator, then
	\begin{enumerate}
	\item[(i)] $R$ is injective,
	\item[(ii)] $\range(R) = \widehat{\mathcal{X}}$,
	\item[(iii)] $\range(R)$ is closed,
	\end{enumerate}
	from which it follows that $R$ is a bijection between $\mathcal{V}$ and $\widehat{\mathcal{X}}$.
	
	If we define the jump operator
	\[
	  B \colon \mathcal{X} \to \mathcal{V}_\Gamma \colon (\vv_C, \vv_I) \mapsto \vv_{C,\tau} - \vv_{I,\tau}
	\]
	then
	\[
	  \range(R) = \ker(B), \qquad \ker(R^T) = \range(B^T).
	\]
\end{corollary}
\begin{proof}
  Property~(i) follows from the fact that if $R \vv = 0$,
	which means that the restriction of $\vv \in \mathcal{V}$ to $\Omega_C$ and to $\Omega_I$ vanish,
	then $\vv = 0$.
	Property~(ii) follows directly from Lemma~\ref{lem:compatSplit}.
	Since the trace operators $\vv_C \mapsto \vv_{C,\tau}$ and $\vv_I \mapsto \vv_{I,\tau}$ are bounded,
	on can see that $\widehat{\mathcal{X}}$ is a closed subspace of $\mathcal{X}$,
	so by Property~(ii), $\range(R)$ is closed.
	The relation $\range(R) = \ker(B)$ is then obvious.
	
	By \cite[Lemma~2.10, Lemma~2.11]{McLean:Book} it follows that $\ker(R^T) = \overline{\range(B)}$, so it remains to show that $\range(B)$ is closed
	(unless all involved spaces are finite-dimensional).
	Firstly, observe that the tangential trace operator from $\mathcal{V}_C$ to $\mathcal{V}_\Gamma$ is surjective and has a bounded right-inverse,
	in particular, it has a closed range.
	The same holds for the tangential trace operator from $\mathcal{V}_I$ to $\mathcal{V}_\Gamma$.
	Secondly, since $B$ simply evaluates the difference of these two traces, it is straightforward to see that
	$\range(B) = \mathcal{V}_\Gamma$: given a function in $\wv_\Gamma \in \mathcal{V}_\Gamma$
	there exists an extension $\wv_C \in \mathcal{V}_C$ such that $\wv_{C,\tau} = \wv_\Gamma$. Setting $\wv_I = 0$,
	we have $B(\wv_C, \wv_I) = \wv_\Gamma$.
	It is then easy to see that $\range(B) = \mathcal{V}_\Gamma$ is closed as well.
	For more details, please refer to \cite[Sect.~4]{Pechstein:2023}.
\end{proof}

\section{Tearing and Interconnecting Formulation}

Let $\Hv(\opcurl,\Omega) \supseteq \mathcal{H} = \nabla \mathcal{U} \oplus \mathcal{V}$
be a splitting, compatible with ($\Omega_C$, $\Omega_I$, $\Gamma$) in the sense of Definition~\ref{def:compatible},
and let $\mathcal{U}_C$, $\mathcal{V}_C$, $\mathcal{V}_I$, and $\mathcal{V}_\Gamma$ be as in that definition.

Let $K_I \colon \mathcal{V}_I \to \mathcal{V}_I^*$ denote the operator
\begin{align}
  \langle K_I \vv_I, \wv_I \rangle := \int_{\Omega_I} \mu^{-1} \opcurl \vv_I \cdot \opcurl \wv_I \, dx \qquad \text{for } \vv_I, \wv_I \in \mathcal{V}_I\,,
\end{align}
and let $K_C \colon \mathcal{V}_C \to \mathcal{V}_C^*$ and $M_C \colon \mathcal{V}_C \to \mathcal{V}_C^*$ be defined by
\begin{align}
  \langle K_C \vv_C, \wv_C \rangle & := \int_{\Omega_C} \mu^{-1} \opcurl \vv_C \cdot \opcurl \wv_C \, dx,\\
	\langle M_C \vv_C, \wv_C \rangle & := \int_{\Omega_C} \sigma \vv_C \cdot \wv_C \, dx,
\end{align}
for $\vv_I$, $\wv_I \in \mathcal{V}_I$.
Moreover, let $S_C \colon \mathcal{V}_C \times \mathcal{U}_C^*$ be defined by
\begin{align}
  \langle S_C \vv_C, q \rangle := \int_{\Omega_C} \sigma \vv_C \cdot \nabla q \, dx \qquad \forall \vv_C \in \mathcal{V}_C, q \in \mathcal{U}_C\,,
\end{align}
and $C_C \colon \mathcal{U}_C \to \mathcal{U}_C^*$ by
\begin{align}
  \langle C_C p, q \rangle := \int_{\Omega_C} \sigma \nabla p \cdot \nabla q \, dx
	\qquad \forall p, q \in \mathcal{U}_C\,.
\end{align}
For the right-hand side, we need more assumptions.
Let $J_C \in \mathcal{V}_C^*$ and $J_I \in \mathcal{V}_I^*$ be current distributions such that
\begin{align}
  \langle J, \vv \rangle_\Omega = \langle J_C, \vv_{|\Omega_C} \rangle_{\Omega_C} + \langle J_I, \vv_{|\Omega_I} \rangle_{\Omega_I}
	\qquad \forall \vv \in \mathcal{V}.
\end{align}
Recall the assumption~\eqref{eq:JsoleniodalI} (or \eqref{eq:JsoleniodalII}) that $\Jv$ is solenoidal outside the conducting domain,
\begin{align}
  \langle \Jv, \nabla q \rangle_\Omega = 0 \qquad \forall q \in \mathcal{U} \text{ with } q_{|\overline{\Omega_C}} = 0.
\end{align}
Suppose, for a moment that $\Jv$ restricted to $\nabla\mathcal{U}$ is assembled from two components $j_C \in \mathcal{U}_C^*$ and $j_I \in \mathcal{U}_I^*$,
such that $\langle \Jv, \nabla q \rangle_\Omega = \langle j_C, q \rangle_{\Omega_C} + \langle j_I, q \rangle_{\Omega_I}$.
Then we can derive from the above that $j_I$ is zero of tested with $q \in \mathcal{U}_I \cap H^1_0(\Omega_C)$ but $j_I$ could very well act on $\Gamma$.
To avoid complications, we refrain from using a separate functional $j_I$ but simply use the functional $j$ defined already in \eqref{eq:jDef}:
\begin{align}
  j \in U_C^* \colon \quad \langle j, q \rangle := \langle \Jv, \nabla(E q) \rangle_\Omega \qquad q \in \mathcal{U}_C
\end{align}
(this definition is independent of the particular extension operator $E \colon \mathcal{U}_C \to \mathcal{U}$).
To summarize, 
\begin{enumerate}
\item[(i)] the complementary component of $\Jv$ is `distributed' to $\Jv_I$ and $\Jv_C$, and this can be quite arbitrary.
\item[(i)] the gradient component of $\Jv$ is \emph{not} distributed but stays with the linear functional $j$, associated with the conducting domain.
\end{enumerate}

\medskip

The tearing and interconnecting formulation reads as follows:

Find $(\Av_I, \Av_C, \varphi, \lambdav) \in \mathcal{V}_I \times \mathcal{V}_C \times \mathcal{U}_C \times \mathcal{V}_\Gamma^*$ such that
\begin{align}
\label{eq:FETI}
  \begin{bmatrix}
	  K_I & 0                         & 0     & B_I^T \\
		0   & K_C + \complexi\omega M_C & S_C^T & B_C^T \\
		0   & \complexi\omega S_C       & C_C   & 0 \\
		B_I & B_C                       & 0     & 0 
	\end{bmatrix}
	\begin{bmatrix} \Av_I \\ \Av_C \\ \varphi \\ \lambdav \end{bmatrix}
	= \begin{bmatrix} J_I \\ J_C \\ j \\ 0 \end{bmatrix},
\end{align}
where
\begin{align}
\begin{aligned}
  B_C \colon \mathcal{V}_C \to \mathcal{V}_\Gamma & \colon \vv_C \mapsto \vv_{C,\tau}\,,\\
  B_I \colon \mathcal{V}_I \to \mathcal{V}_\Gamma & \colon \vv_I \mapsto -\vv_{I,\tau}\,.
\end{aligned}
\end{align}

\begin{theorem}
  Let the geometric assumptions from Sect.~\ref{sect:geoAss} hold and let
  $\Hv(\opcurl,\Omega) \supseteq \mathcal{H} = \nabla \mathcal{U} \oplus \mathcal{V}$
	be a direct splitting that is compatible with $(\Omega_C$, $\Omega_I$, $\Gamma)$ in the sense of Definition~\ref{def:compatible}.
  Then formulations~\eqref{eq:MQS:operator} and \eqref{eq:FETI} are equivalent.
\end{theorem}
\begin{proof}
  Any solution of~\eqref{eq:FETI} fulfills $B_I \Av + B_C \Av = 0$, so $(\Av_I, \Av_C) \in \widehat{\Xv}$
  and according to Corollary~\ref{cor:X} there exists a unique function $\Av \in \mathcal{V}$ such that
  $\Av_I = \Av_{|\Omega_I}$ and $\Av_C = \Av_{|\Omega_C}$.
  Once this is established and one tests the formulation with $(\vv_{|\Omega_C}, \vv_{|\Omega_I})$ for $\vv \in \mathcal{V}$,
  the Lagrange multiplier terms cancel out and we are left with \eqref{eq:MQS:operator}.

  Let $(\Av, \varphi)$ be a solution of \eqref{eq:MQS:operator} and let $R \colon \mathcal{V} \to \mathcal{X}$
	be the tearing operator from Corollary~\ref{cor:X}.
	Then, by the stated assumptions, we have the relations
	\begin{align}
	\begin{alignedat}{2}
	  K & = R^T \begin{bmatrix} K_C & 0 \\ 0 & K_I \end{bmatrix} R, & \qquad
	  M & = R^T \begin{bmatrix} M_C & 0 \\ 0 & 0 \end{bmatrix} R, \\
		S & = \begin{bmatrix} S_C & 0 \end{bmatrix} R, & \qquad
		J & = R^T \begin{bmatrix} J_C \\ J_I \end{bmatrix}.
	\end{alignedat}
	\end{align}
	With that, we can reformulate~\eqref{eq:MQS:operator} as
	\begin{align}
	  R^T \left( \begin{bmatrix} J_C \\ J_I \end{bmatrix}
		  - \begin{bmatrix} K_C + \complexi\omega M_C & 0 \\ 0 & K_I \end{bmatrix} R \Av - \begin{bmatrix} S_C^T \\ 0 \end{bmatrix} \varphi \right)
		& = 0,\\
		\begin{bmatrix} S_C & 0 \end{bmatrix} R \Av + C \varphi & = j.
	\end{align}
	We define $(\Av_C, \Av_I) \in \mathcal{X}$ by
	\begin{align}
	  \begin{bmatrix} \Av_C \\ \Av_I \end{bmatrix} := R \Av
	\end{align}
	and $(\tv_C, \tv_I) \in \mathcal{X}^*$ by
	\begin{align}
	  \begin{bmatrix} \tv_C \\ \tv_I \end{bmatrix}
		= \begin{bmatrix} J_C \\ J_I \end{bmatrix}
		  - \begin{bmatrix} K_C + \complexi\omega M_C & 0 \\ 0 & K_I \end{bmatrix} R \Av - \begin{bmatrix} S_C^T \\ 0 \end{bmatrix} \varphi.
	\end{align}
	Therefore, we have
	\begin{align}
	  R^T \begin{bmatrix} \tv_C \\ \tv_I \end{bmatrix} & = 0,\\
	  \begin{bmatrix}
	    K_I & 0                         & 0     \\
	  	0   & K_C + \complexi\omega M_C & S_C^T \\
  		0   & \complexi\omega S_C       & C_C     
  	\end{bmatrix}
	  \begin{bmatrix} \Av_I \\ \Av_C \\ \varphi \end{bmatrix}
	  + \begin{bmatrix} \tv_I \\ \tv_C \\ 0 \end{bmatrix}
	  & = \begin{bmatrix} J_I \\ J_C \\ j \end{bmatrix},\\
	\label{eq:ACIrangeR}
		\begin{bmatrix} \Av_C \\ \Av_I \end{bmatrix} & \in \range(R).
  \end{align}
	By Corollary~\ref{cor:X}, $\range(R) = \ker(B)$ and $\ker(R^T) = \range(B^T)$.
	Therefore, we can find $\lambdav \in \mathcal{V}_\Gamma^*$ depending continuously on $(\tv_C, \tv_I$)
	such that
	\begin{align}
	  \begin{bmatrix} \tv_C \\ \tv_I \end{bmatrix} = B^T \lambdav = \begin{bmatrix} B_C^T \\ B_I^T \end{bmatrix} \lambdav,
	\end{align}
	and the condition~\eqref{eq:ACIrangeR} is equivalent to
	\begin{align}
	  B \begin{bmatrix} \Av_C \\ \Av_I \end{bmatrix}
		= \begin{bmatrix} B_C & B_I \end{bmatrix} \begin{bmatrix} \Av_C \\ \Av_I \end{bmatrix}
		= 0.
	\end{align}
	It is now easy to see that $(\Av_I, \Av_C, \varphi, \lambdav)$ solves \eqref{eq:FETI}.
	For more background information see also \cite[Sect.~4]{Pechstein:2023}.
\end{proof}

This note is concluded with the following remarks.
\begin{enumerate}
\item[(i)]
  Since the subspace $\mathcal{V}_I$ does not contain any gradient fields and $\Omega_I$ is simply connected,
  the local subdomain operator $K_I$ is invertible.
\item[(ii)]
  If $(\Av_I, \Av_C)$ is a solution of~\eqref{eq:FETI}, then according to Corollary~\ref{cor:curl},
  the vector field $\Bv \in \Lv^2(\Omega)$ defined by
  \[
    \Bv(\xv) := \begin{cases} 
	              (\opcurl\Av_I)(\xv) & \xv \in \Omega_I, \\
								(\opcurl\Av_C)(\xv) & \xv \in \Omega_C
							\end{cases}
  \]
  is globally in $\Hv(\opdiv,\Omega)$.
\end{enumerate}

\bibliographystyle{abbrv}
\bibliography{IETI-EddyCurrent}

\end{document}